\documentclass[a4paper]{amsart}
\usepackage{txfonts, amsmath,amstext,amsthm,amscd,amsopn,verbatim,amssymb, amsfonts}
\usepackage{fullpage}

\usepackage[bbgreekl]{mathbbol}
\usepackage{tikz}
\usetikzlibrary{matrix}
\usetikzlibrary{shapes}
\usetikzlibrary{arrows}
\usetikzlibrary{calc,3d}
\usetikzlibrary{decorations,decorations.pathmorphing}
\usetikzlibrary{through}
\tikzset{ext/.style={circle, draw,inner sep=1pt},int/.style={circle,draw,fill,inner sep=1pt},nil/.style={inner sep=1pt}}
\tikzset{exte/.style={circle, draw,inner sep=3pt},inte/.style={circle,draw,fill,inner sep=3pt}}
\tikzset{diagram/.style={matrix of math nodes, row sep=3em, column sep=2.5em, text height=1.5ex, text depth=0.25ex}}
\tikzset{diagram2/.style={matrix of math nodes, row sep=0.5em, column sep=0.5em, text height=1.5ex, text depth=0.25ex}}
\theoremstyle{plain}
  \newtheorem{thm}{Theorem}
  
  \newtheorem{prop}{Proposition}
  
  \newtheorem{cor}[prop]{Corollary}
  \newtheorem{lemma}{Lemma}
\theoremstyle{definition}
  \newtheorem{ex}{Example}
  \newtheorem{rem}{Remark}

\newcommand{\alg}[1]{\mathfrak{{#1}}}
\newcommand{\R}{{\mathbb{R}}}
\newcommand{\Z}{{\mathbb{Z}}}
\newcommand{\K}{{\mathbb{K}}}

\newcommand{\fGCc}{{\mathrm{fcGC}}}

\newcommand{\Graphs}{{\mathsf{Graphs}}}

\newcommand{\te}{{\tilde{\mathsf{e}}}}

\newcommand{\fcGraphs}{{\mathsf{fcGraphs}}}
\newcommand{\dfcGraphs}{{\mathsf{dfcGraphs}}}
\newcommand{\fcGraphso}{\mathsf{fcGraphs}^{or}}

\newcommand{\Graphso}{\mathsf{Graphs}^{or}}

\newcommand{\Def}{\mathrm{Def}}

\newcommand{\op}{\mathcal}

\newcommand{\Lie}{\mathsf{Lie}}

\newcommand{\hoLie}{\mathsf{hoLie}}

\newcommand{\Com}{\mathsf{Com}}

\newcommand{\Der}{\mathrm{Der}}

\newcommand{\conn}{\mathit{conn}}

\newcommand{\bpm}{\begin{pmatrix}}
\newcommand{\epm}{\end{pmatrix}}
\newcommand{\Tpoly}{T_{\rm poly}}
\newcommand{\Dpoly}{D_{\rm poly}}

\newcommand{\GC}{\mathrm{GC}}
\newcommand{\GCo}{\mathrm{GC}^{or}}

\newcommand{\fGC}{\mathrm{fGC}}
\newcommand{\fGCo}{\mathrm{fGC}^{or}}
\newcommand{\fGCco}{\mathrm{fcGC}^{or}}
\newcommand{\dfGCc}{\mathrm{dfcGC}}
\newcommand{\dfGC}{\mathrm{dfGC}}
\newcommand{\GCor}{\mathrm{GC}^{or}}
\newcommand{\hGCor}{\widehat{\mathrm{GC}}^{or}}

\DeclareMathOperator{\End}{End}
\DeclareMathOperator{\sgn}{sgn}

\newcommand{\grt}{\alg {grt}}
\newcommand{\hoe}{\mathsf{hoe}}
\newcommand{\e}{\mathsf{e}}

\newcommand{\dimens}{\mathop{dim}}
\newcommand{\gra}{\mathrm{gra}}

\begin{document}
\title{The oriented graph complexes}
\author{Thomas Willwacher}
\address{Department of Mathematics\\ ETH Zurich\\ R\"amistrasse 101 \\ 8092 Zurich, Switzerland}
\email{t.willwacher@gmail.com}

%\thanks{The author was partially supported by the Swiss National Science Foundation (grant 200020-105450).}
% \subjclass[2000]{16E45; 53D55; 53C15; 18G55}
% \date{}
%\keywords{Formality, Deformation Quantization, Operads}

\begin{abstract}
Oriented graph complexes, in which graphs are not allowed to have oriented cycles, govern for example the quantization of Lie bialgebras and infinite dimensional deformation quantization.
It is shown that the oriented graph complex $\GCor_n$ is quasi-isomorphic to the ordinary commutative graph complex $\GC_{n-1}$, up to some known classes.
This yields in particular a combinatorial description of the action of $\grt_1\cong H(\GC_2)$ on Lie bialgebras, and shows that a cycle-free formality morphism in the sense of Shoikhet can be constructed rationally without reference to configuration space integrals.
 %We also give a configuration space integral version of Tamarkin's quantization of Lie bialgebras via the formality of the little disks operad.
\end{abstract}
\maketitle

\section{Introduction}
Graph complexes are differential graded vector spaces whose elements are linear combinations or series of isomorphism classes of graphs. 
Various flavors of graph complexes exist, depending on the type of graphs that are allowed in the series. The most commonly encountered complexes are the ribbon graph complexes, which compute the cohomology of moduli spaces of curves, the ``Lie'' graph complex which computes the cohomologies of the automorphism groups of free groups \cite{cullervogtmann}, and the ``commutative'' graph complexes which govern the deformations of the $E_n$ operads \cite{grt}.
One may also define various other versions of graph complexes. For example, one can define a graph complex starting from any cyclic operad and/or by allowing for external legs, possibly with some decoration; see \cite{LambrechtsTurchin} for an application to low dimensional topology and also \cite{vogtmannhairy}.

The common feature of all these graph complexes is that their cohomology is very hard to compute, and usually only very few facts are known beyond the computer accessible regime.

In this paper we will study a version of graph complexes first introduced (to the knowledge of the author) by S. Merkulov, the oriented graph complexes $\GC_n^{or}$, whose elements are linear combinations of isomorphism classes of directed graphs, which do not contain oriented loops. In other words, the directions of the edges naturally endow the set of vertices with the structure of a partially ordered set. For a more precise definition of these complexes see section \ref{sec:graphcomplexes}. 
The complexes $\GC_n^{or}$ carry a natural dg Lie algebra structure.

These oriented graph complexes appear for example in the quantization of Lie bialgebras (for $n=3$), where they act on Lie bialgebra structures, and in infinite dimensional deformation quantization (for $n=2$), see also \cite{shoikhetlinfty}.

Slightly different versions of oriented graph complexes with external legs have been considered in \cite{MarklVoronov}. These complexes are highly related, for an indication of the link see section \ref{sec:beardGC}.

The main result of this paper is a computation of the cohomology of the oriented graph complexes.
\begin{thm}\label{thm:main}
 The cohomology of the oriented graph complex $\GC_n^{or}$ is isomorphic to the cohomology of the ordinary commutative graph complex $\fGCc_{n-1}$ as dg Lie algebra. In particular
\[
 H(\GC_n^{or})\cong H(\fGCc_{n-1})=H(\GC_{n-1})\oplus \bigoplus_{\substack{j\geq 1 \\ j\equiv 2n+1 \text{ mod }4 }}
 \K[n-j].
\]
Furthermore, the identification preserves the additional grading by the first Betti numbers of graphs on both sides.
\end{thm}
For a definition of the complexes involved, see section \ref{sec:graphcomplexes}.
During the proof of the Theorem we will also give a method for mapping graph cocycles in the commutative graph complex to cocycles in $\GC_n^{or}$ or vice versa.

There are two main applications of this result that we are aware of. 
First, for $n=3$ degree 0 graph cocycles in $\GC_n^{or}$ act naturally on Lie bialgebra structures. Since $H^0(\GC_3^{or})=H^0(\fGCc_{2})=\grt_1$ is the Grothendieck-Teichm\"uller Lie algebra (see \cite{grt}), we recover the action of $\grt_1$ on Lie bialgebra structures. This action had been constructed before, but through a relatively complicated and inexplicit argument.
Theorem \ref{thm:main} above yields a relatively explicit description, in particular since explicit integral formulas for the graph cocycles corresponding to Deligne-Drinfeld elements (and conjectural generators) $\sigma_3, \sigma_5,\dots\in \grt_1$ are known, see \cite{carlothomas}.

The second application of Theorem \ref{thm:main} is in infinite dimensional deformation quantization. In finite dimensions the deformation quantization problem has been solved by M. Kontsevich by proving his Formality Conjecture, asserting the existence of an $L_\infty$ quasi-ismorphism
\[
 \Tpoly \to \Dpoly
\]
between the Lie algebra of multivector fields $\Tpoly$ on some manifold $M$ and the differential graded Lie algebra of multidifferential Hochschild cochains $\Dpoly$, for details see \cite{K1}.
When introducing the conjecture \cite{Kgrcomplex} Kontsevich noted that the obstruction theory of such a formality morphism\footnote{To be precise, in the local case $M=\R^d$ and for formality morphism given by universal formulas. } is governed by the Lie algebra $\GC_2\subset \fGCc_2$. In particular, obstructions to the existence of a formality morphism fall in $H^1(\GC_2)$. It is a famous conjecture, called the Drinfeld-Kontsevich conjecture in this paper, that $H^1(\GC_2)=0$. Unfortunately, this conjecture has resisted significant efforts by many people over the last 20 years. M. Kontsevich finally proved his formality conjecture through a different route by an explicit construction using configuration space integrals. V. Drinfeld, who faced the equivalent problem of constructing a Drinfeld associator also had to resort to transcendental methods and constructed the Knizhnik-Zamolodchikov associator.

M. Kontsevich's formality morphism cannot be used in the infinite dimensional setting. The problem is that graphs he used in the definition may contain directed cycles, which lead to divergences.
The infinite dimensional deformation quantization problem was treated by B. Shoikhet \cite{shoikhetlinfty}. He showed that there is an obstruction to constructing a formality morphism given by universal formulas in infinite dimensions.
This obstruction is implicit also in the earlier work of Penkava and Vanhaecke \cite{penkavavanhaecke}.
However, Shoikhet showed that the $L_\infty$ structure on $\Tpoly$ may be modified so as to obtain an $L_\infty$ algebra $\Tpoly'$, such that there exists a formality morphism
\[
 \Tpoly' \to \Dpoly.
\]
Shoikhet uses configuration space integrals similar to Kontsevich's to construct both the modified $L_\infty$ structure on $\Tpoly$ and the formality morphism.

Theorem \ref{thm:main} above allows one to fully understand the obstruction theory underlying Shoikhet's construction.
First, obstructions to the existence of a formality morphism $\Tpoly\to \Dpoly$ by universal formulas without loops fall into $H^1(\GC_2^{or})\cong H^1(\GC_1)\cong \K \theta$.
On the right hand side the single class corresponds to the $\theta$-graph 
\[
 \theta = 
\begin{tikzpicture}
 [baseline=-.65ex]
 \node[int] (v) at (0,0) {};
 \node[int] (w) at (1,0) {};
\draw (v) edge[bend left] (w) edge[bend right] (w) edge (w);
%\node at (.5,-.5) {$j$ edges};
\end{tikzpicture},
\]
hence the notation. The corresponding graph cocycle in $\GC_2^{or}$ is the obstruction encountered by Penkava-Vanhaecke and Shoikhet, and is depicted in Example \ref{ex:shoiobstruction} below. 
The set of $L_\infty$ structures on $\Tpoly$ given by universal formulas without loops is also governed by $\GC_2^{or}$, and the fact that $H^1(\GC_2^{or})\cong \K$ shows that they form a one-dimensional space. The one dimension is exhausted by more or less trivial rescalings, so up to those the Shoikhet $L_\infty$ structure is unique.

Once an appropriate $L_\infty$ structure is fixed, determined by a Maurer-Cartan element $m\in \GC_2^{or}$, the existence of a formality morphism is unobstructed.
% $H^1((\GC_2^{or})^m)=H^1(\GC_1^\Theta)$, where $\Theta\in \GC_1$ is the Maurer-Cartan element corresponding to $m$. Concretely one finds (TODO:check prefactors)
%\[
%\Theta =
%\sum_{j\geq 0}
%\frac 1 {j!}
%\begin{tikzpicture}[baseline=-.65ex]
% \node[ext] (v) at (0,0) {1};
% \node[ext] (w) at (1,0) {2};
%\draw (v) edge[bend left] (w) edge[bend right] (w) edge (w);
%\node at (.5,-.5) {$j$ edges};
%\end{tikzpicture}.
%\]
%It is shown in \cite{markothomas} that $H(\GC_1^\Theta)=0$.\footnote{In the relevant degrees 0 and 1 this can also be inferred easily by degree reasons.} Hence all further obstructions to the existence of a formality morphism vanish, and up to homotopy there is a unique such morphism. 
Let us summarize these results.

\begin{cor}\label{cor:looplessformality}
 Shoikhet's $L_\infty$ structure on $\Tpoly$ is essentially the unique (up to rescalings and gauge) $L_\infty$ structure given by loop-less universal formulas.
The existence of a formality morphism by loopless universal formulas is unobstructed and such a morphism is unique up to homotopy. All constructions of these objects can be done rationally, without appeal to transcendental methods.
\end{cor}

Note that this is in sharp contrast to the standard setting, where the obstruction theory, i.~e., both obstructions to and choices of formality morphisms, are subject to hard conjectures,\footnote{The Drinfeld-Kontsevich conjecture $H^1(\GC_2)\stackrel{?}{=}0$ for the obstructions and the Deligne-Drinfeld-Ihara conjecture $\grt_1\stackrel{?}{=}\widehat{\mathrm{FreeLie}}(\sigma_3, \sigma_5,\dots)$ for the choices.} and all known constructions of formality morphisms involve transcendental methods.

\subsection*{Structure of the paper}

In section \ref{sec:graphcomplexes} we recall the definitions of the relevant objects, the graph complexes $\GC_n^{or}$ and $\GC_n$, and related operads that will appear in the proof of Theorem \ref{thm:main} in section \ref{sec:theproof}.
The application to 
%Lie bialgebras ($n=3$) and to 
infinite dimensional deformation quantization will be discussed in more detail in section %s \ref{sec:liebialgebras} and 
\ref{sec:looplessdefq}.

%Finally, the appendix contains a brief derivation of the obstruction to the existence of a star product given by loopless universal formulas.

\subsection*{Acknowledgements}
I am grateful to S. Merkulov for many discussions. It was also him who suggested the problem. The author was partially supported by the Swiss National Science foundation, grant PDAMP2\_137151.

\section{Notation and basic definitions}
We fix a ground field $\K$ of characteristic zero. We abbreviate the phrase ``differential graded'' by dg as usual.

We will use freely the language of operads. A good introduction can be found in the book \cite{lodayval}.
In particular, we will use the operads $\e_n$, $n\in \Z$. An algebra over the operad $\e_n$ is a graded vector space together with an associative commutative product $\wedge$ and a Lie bracket operation $[,]$ of degree $1-n$ such that for all homogeneous $x,y,z$ in the vector space
\[
 [x,y\wedge z]=[x,y]\wedge z + (-1)^{(|x|+1-n)|y|} y\wedge [x,z].
\]
One can check in particular that $\dimens \e_n(N)=N!$.
In the literature these operads are sometines denoted by $P_n$ since $\e_1$ is the Poisson operad.
They are Koszul, with Koszul dual $\e_n^\vee = \e_n^*\{n\}$.
We call the minimal resolution $\hoe_n:= \Omega(e_n^\vee)$, where $\Omega(\dots)$ denotes the operadic cobar construction.
There is a canonical map $\hoe_n\to e_n$.
We will denote by $\hoLie_n:=\Omega((\Lie\{n-1\})^\vee)$ the minimal resolution of the degree shifted Lie operad. There is an inclusion $\hoLie_n\to \hoe_n$.

We also need the operads $\te_n$, which are quotients of $\hoe_n$. An $\te_n$ algebra is a $\hoe_n$ algebra whose only non-vanishing $\hoe_n$ operations are (i) the commutative product operation and (ii) the $\hoLie_n$ operations.
In other words, a $\te_n$ algebra is a commutative algebra with a degree shifted $L_\infty$ structure, such that the $L_\infty$ operations are derivations in each slot.

\begin{rem}
 One way the operad $\te_2$ arises in practice is the following.
Let $\alg g$ be a Lie bialgebra. Then the co-Chevalley complex $S\alg g[-1]$ is naturally an $\e_2$ algebra. Here the Lie algebra structure determines the bracket and the Lie coalgebra structure determines the differential.

Similar, if $\alg g$ is an $\infty$-Lie bialgebra, then $S\alg g[-1]$ is naturally a $\hoe_2$ algebra. In fact, in this case the action of $\hoe_2$ factorizes as follows:
\[
 \hoe_2 \to \te_2 \to \End(S\alg g[-1]).
\]
\end{rem}

Below we will need the following result.
\begin{prop}\label{prop:tenenqiso}
 The map $\te_n\to \e_n$ is a quasi-isomorphism.
\end{prop}
\begin{proof}
As a $\mathbb{S}$-module $\te_n\cong \Com\circ \hoLie_n$, hence the cohomology is $\Com\circ \Lie\{n-1\}\cong \e_n$.
\end{proof}

\begin{rem}
 For a quasi-Lie bialgebra or $\infty$-quasi-Lie bialgebra, the Chevalley complex is in general not a $\te_2$ algebra, only a ``non-flat'' $\te_2$ algebra, meaning that the Chevalley differential does not square to zero. Such a non-flat $\te_2$ algebra may be seen as an algebra (with zero differential) over an operad $\te_2^{nonflat}$ defined in a similar way as $\te_2$, except that one allows in addition zero- and unary operations $\mu_0, \mu_1$. The operations $\mu_0,\mu_1, \mu_2,\dots$ then generate a suboperad isomorphic (up to a degree shift) to the operad governing non-flat $L_\infty$ algebras.
There is a map $\te^{nonflat}_2\to \te_2$ sending $\mu_0$ and $\mu_1$ to zero. However, note that $\te_2^{nonflat}$ is acyclic.
\end{rem}

For $\op P$ an operad, we may consider the dg Lie algebra of derivations $\Der(\op P)$ of this operad. Similarly, for $\op P \to \op Q$ an operad map, we may consider the deformation complex $\Def(\op P \to \op Q)$. In fact, in the cases we encounter, we may always assume that $\op P$ has the form $\op P=\Omega(\op C)$ for a coaugmented cooperad $\op C$. For the precise definition of $\Der(dots)$ and $\Def(dots)$ we refer the reader to \cite[section 2]{grt}, whose conventions we follow.

We will also make use of the following small result, contained in loc. cit.
\begin{lemma}[Lemma 2.2 of \cite{grt}]
\label{lem:defpresqiso}
Let $\op C$ be a coaugmented cooperad, let $\op P$ and $\op P'$ be operads, and let 
\[
\Omega(\op C) \to \op P \to \op P'
\]
be operad maps, with the right hand arrow being a quasi-isomorphism. Then the induced map of differential graded Lie algebras
\[
\Def(\Omega(\op C)\to \op P) \to \Def(\Omega(\op C)\to \op P')
\]
is a quasi-isomorphism.
\end{lemma}

\section{Graph complexes and graph operads}\label{sec:graphcomplexes}
\subsection{The ordinary commutative graph complex}
We briefly recall the construction of the commutative graph complexes, for more details the reader is referred to \cite{grt}.
Let $\gra_{N,k}$ be the set of directed graphs with vertex set $[n]=\{1,\dots, n\}$ and edge set $[k]$. It carries an action of the group $G_{N,k}:=S_N\times (S_2^k \ltimes S_k)$ by renumbering the vertices, renumbering the edges and flipping the directions of edges.
As a graded vector space one may define the full graph complex to be
\begin{equation}\label{equ:fGCndef}
 \fGC_n 
=
\prod_{N\geq 1}
\prod_{k\geq 0}
\left(
\K\langle \gra_{N,k} \rangle \otimes \K[-n]^{\otimes N} \otimes \K[n-1]^{\otimes k} \otimes \sgn^{\otimes k}
\right)_{G_{N,k}}[n].
\end{equation}
Here $G_n$ acts diagonally on the vector space spanned by graphs $\K\langle \gra_{N,k} \rangle$ and on the one dimensional tensor products on the right by permutations, i.~e., by appropriate signs. There is a pre-Lie algebra structure $\bullet$ on $\fGC_n$ such that for graphs $\gamma, \nu$ the pre-Lie product $\gamma\bullet\nu$ is the sum over all possible insertions of $\nu$ at vertices of $\gamma$.
One checks that the element 
\[
m=
\begin{tikzpicture}[baseline=-.65ex]
 \node[int] (v) at (0,0) {};
\node[int] (w) at (0.5,0) {};
\draw (v) edge (w);
\end{tikzpicture}
\]
is a Maurer-Cartan element, i.~e., $[m,m]=0$. The differential on the graph complex $\fGC_n$ is the bracket with $m$, $\delta=[m,\cdot]$. Combinatorially, this differential is given by vertex splitting.

There are two important subcomplexes $\GC_n \subset \fGCc_n\subset \fGC_n$. Here $\fGCc_n$ consists of series of connected graphs, and $\GC_n$ consists of series in connected graphs all of whose vertices are at least trivalent.
The cohomology of $\fGC_n$ is clearly just the symmetric product of that of $\fGCc_n$. Furthermore, the latter cohomology is the sum of of $\GC_n$ and the wheel classes, which are represented by graphs which are loops of bivalent vertices.

\begin{rem}
 The author apologizes for the clumsy notation like $\fGCc_n$. The $\GC$ shall stand for graph complex, the $\mathrm{f}$ for full, i.~e., allowing all valences of vertices, and $\mathrm{c}$ for connected, i.~e., graphs are only allowed to have one connected component. Later we will also see a $\mathrm{d}$, which shall indicate that the edges are directed.
\end{rem}

\begin{rem}
 In this paper we take the approach that all graphs are allowed to contain short cycles a priori, i.~e., edges connecting some vertex to itself.
This is not consistent with the notation used elsewhere. However, we don't want to spoil the already complicated notation further by putting ${}^\circlearrowleft$ superscripts everywhere. 
\end{rem}

\subsection{The oriented graph complex}
 There is also a directed graph complex $\dfGC_n$, built like $\fGC_n$ except that one retains the direction of edges, i.~e., one does not mod out by $S_2^k$ in formula \eqref{equ:fGCndef} above. One can however check (see Appendix K in \cite{grt}) that both complexes $\dfGC_n$ and $\fGC_n$ are quasi-isomorphic, so nothing new is created.
However, the directed graph complex has an interesting subcomplex
\[
 \fGC_n^{or} \subset \dfGC_n
\]
whose elements are series in graphs that do not possess directed cycles. 
Similarly to the above discussion, we may also identify subcomplexes 
\[
 \GC_n^{or}\subset \fGCc_n^{or} \subset \fGC_n^{or}
\]
where $\fGCc_n^{or}$ consists of series of \emph{connected} graphs without oriented cycles, and $\GC_n^{or}$ furthermore contains only graphs all of whose vertices have valence at least 2.

\begin{ex}\label{ex:shoiobstruction}
 The non-trivial class in $\GC_n^{or}$ ($n$ even) with the fewest vertices is represented by the following cocyle, found by B. Shoikhet \cite{shoikhetlinfty}, and implicitly by Penkava and Vanhaecke \cite{penkavavanhaecke}.
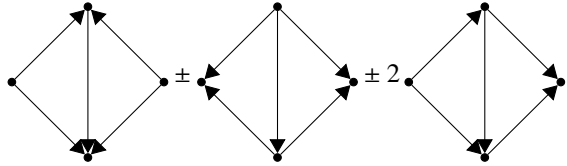
\[
\begin{tikzpicture}[baseline={(current bounding box.center)},every edge/.style={draw, -triangle 60}]
\node[int] (v1) at (-1,1.5) {};
\node[int] (v2) at (-1,-0.5) {};
\node[int] (v4) at (-2,0.5) {};
\node[int] (v3) at (0,0.5) {};
\draw  (v1) edge (v2);
\draw  (v3) edge (v2);
\draw  (v3) edge (v1);
\draw  (v4) edge (v1);
\draw  (v4) edge (v2);
\end{tikzpicture}
\pm
\begin{tikzpicture}[baseline={(current bounding box.center)},every edge/.style={draw, -triangle 60}]
\node[int] (v1) at (-1,1.5) {};
\node[int] (v2) at (-1,-0.5) {};
\node[int] (v4) at (-2,0.5) {};
\node[int] (v3) at (0,0.5) {};
\draw  (v1) edge (v2);
\draw  (v2) edge (v3);
\draw  (v1) edge (v3);
\draw  (v1) edge (v4);
\draw  (v2) edge (v4);
\end{tikzpicture}
\pm 2 \,
\begin{tikzpicture}[baseline={(current bounding box.center)},every edge/.style={draw, -triangle 60}]
\node[int] (v1) at (-1,1.5) {};
\node[int] (v2) at (-1,-0.5) {};
\node[int] (v4) at (-2,0.5) {};
\node[int] (v3) at (0,0.5) {};
\draw  (v1) edge (v2);
\draw  (v1) edge (v3);
\draw  (v2) edge (v3);
\draw  (v4) edge (v1);
\draw  (v4) edge (v2);
\end{tikzpicture}
\]
\end{ex}

\begin{rem}
 Degree 0 cocycles in the oriented graph complex $\GCo_3$ act naturally on the space of Lie bialgebra structures on any (possibly infinite dimensional) vector space.
More generally, there is a map from $\GCo_3$ to the (bi-)Chevalley complex of any Lie bialgebra.
\end{rem}

%We will need the following result below.
%
%\begin{prop}
%The inclusion $\GCo_n\subset \fGCco_n$ is a quasi-isomorphism.
%\end{prop}
%\begin{proof}
%The proof is a copy of Proposition 3.4 of \cite{grt}.
%\end{proof}

\subsection{Another (equivalent) version of the oriented graph complex}
\label{sec:beardGC}
We will in fact encounter one more version of the oriented graph complex in the proof of Theorem \ref{thm:main} below. Define $\hGCor_n$ in the same manner as $\fGCco_n$, except that the class of graphs one considers is larger. One allows the graphs to have any number of outgoing ``external legs'', and one sets graphs to zero if they contain vertices without outgoing edges. The external legs do not alter the degree of the graph, in other words they are considered to carry degree 0.

\begin{ex}
Here are graphs cochains in (left to right) $\hGCor_3$ and $\GCor_3$:
\[
%\begin{tikzpicture}[baseline=-0.65ex]
% \node[int] (v2) at (0,0) {};
% \node[int] (v1) at (0,1) {};
%  \draw[-triangle 60] (v1) edge[bend left] (v2) edge[bend right] (v2);
% \node (w1) at (-.7,-1) {};
%\node (w2) at (-.3,-1) {};
%\node (w3) at (.3,-1) {};
%\node (w4) at (.7,-1) {};
%\draw[-triangle 60] (v2) edge (w1) edge (w2) edge (w3) edge (w4);
% \node (x1) at (-.7,2) {};
%\node (x2) at (-.3,2) {};
%\node (x3) at (.3,2) {};
%\node (x4) at (.7,2) {};
%\draw[triangle 60-] (v1) edge (x1) edge (x2) edge (x3) edge (x4);
%\end{tikzpicture}
\begin{tikzpicture}[baseline=-0.65ex]
 \node[int] (v2) at (0,0) {};
 \node[int] (v1) at (0,1) {};
 \draw[-triangle 60] (v1) edge[bend left] (v2) edge[bend right] (v2);
 \node (w1) at (-.7,-1) {};
\node (w2) at (-.3,-1) {};
\node (w3) at (.3,-1) {};
\node (w4) at (.7,-1) {};
\draw[-triangle 60] (v2) edge (w1) edge (w2) edge (w3) edge (w4);
\end{tikzpicture}
\begin{tikzpicture}[baseline=-0.65ex]
 \node[int] (v2) at (0,0) {};
 \node[int] (v1) at (0,1) {};
 \draw[-triangle 60] (v1) edge[bend left] (v2) edge[bend right] (v2);
\end{tikzpicture}
\]
The right hand graph is zero as an element of $\hGCor_3$ because it has a vertex without outgoing edges.
\end{ex}

%
% as the complex of series of connected directed graphs with external outgoing edges, such that
%\begin{itemize}
% \item There is no directed loop.
% \item All vertices have at least one outgoing edge and one incoming edge.
% \item The vertices have degree +3, internal edges have degree -2 and external edges have degree 0.
%\end{itemize}
%
%There is a natural differential and two different Lie brackets (TODO: describe).
%
%Next consider a complex $\GCor_3$ defined in the same manner except that 
%\begin{itemize}
% \item There are no external edges.
% \item Vertices with no incoming or outgoing edges are allowed.
% \item All vertices must have valence at least 2.
%\end{itemize}

The differential on $\hGCor_n$ is pictorially defined as follows.
\begin{equation}\label{equ:hGCodifferential}
 \delta \Gamma = \sum_\nu \Gamma \bullet_\nu
\begin{tikzpicture}[baseline=-.65ex]
 \node[int] (v) at (0,.3) {};
\node[int] (w) at (0,-.3) {};
\draw[-triangle 60] (v) edge (w);
\end{tikzpicture}
\pm
\sum_{j\geq 1}
\frac 1 {j!}
\begin{tikzpicture}[baseline=-.65ex]
 \node[int] (v) at (0,.3) {};
\node[ext] (w) at (0,-.3) {$\Gamma$};
\coordinate (w1) at (+.3,-.1);
\coordinate (w2) at (-.3,-.1);
\draw[-triangle 60] (v) edge (w) edge (w1) edge (w2);
\end{tikzpicture}
\pm
\sum_{j\geq 1}
\frac 1 {j!}
\begin{tikzpicture}[baseline=-.65ex]
 \node[ext] (v) at (0,.3) {$\Gamma$};
\node[int] (w) at (0,-.3) {};
\coordinate (w1) at (+.3,-.5);
\coordinate (w2) at (-.3,-.5);
\coordinate (w3) at (0,-.7);
\draw[-triangle 60] (v) edge (w) (w) edge (w1) edge (w2) edge (w3);
\end{tikzpicture}
\end{equation}
Here the first sum runs over vertices of $\Gamma$, and the symbol $\bullet_\nu$ shall mean that the graph on the right is inserted at vertex $\nu$. In the second term the black vertex has valence $j$, and one sums in addition over all ways to connect the edge towards $\Gamma$ to vertices of $\Gamma$. In the last term one sums over all external legs of $\Gamma$ and connects one to the new vertex.

There is a map of complexes
\[
 \GCo_n \to \hGCor_n
\]
pictorially defined as follows
\begin{equation}\label{equ:hairymap}
\Gamma \mapsto 
%\sum_{i=1}^\infty 
\sum_{j=1}^\infty
\frac 1 {j!}
\underbrace{
%\overbrace{
 \begin{tikzpicture}[baseline=-.65ex]
  \node (v) at (0,0) {$\Gamma$};
  \coordinate (v0) at (-.7,-1);
  \coordinate (v1) at (-.3,-1);
  \coordinate (v2) at (.3,-1);
  \coordinate (v3) at (.7,-1);
  \node (w0) at (-.7,1) {};
  \node (w1) at (-.3,1) {};
  \node (w2) at (.3,1) {};
  \node (w3) at (.7,1) {};
  \draw[-triangle 60] (v) edge (v0) edge (v1) edge (v2) edge (v3);  
  \%draw[-triangle 60] (w0) edge (v) (w1) edge (v) (w2) edge (v) (w3) edge (v);  
 \end{tikzpicture}
% }^{i\times}}
 }_{j\times}
\end{equation}
where the picture on the right means that one should sum over all ways of connecting $j$ outgoing edges to the graph $\Gamma$. Graphs for which there remain vertices with no outgoing edge are identified with $0$.

\begin{prop}\label{prop:GChGC}
 The map $\GCo_n \to \hGCor_n$ is a quasi-isomorphism up to the class in $H(\hGCor_n)$ represented by the graph cocycle
\begin{equation}
\label{equ:singleclass}
 \sum_{j\geq 2} \frac{j-1}{j!}
 \underbrace{\begin{tikzpicture}[baseline=-2.5ex, scale=.7]
  \node[int] (v) at (0,0) {};
 \coordinate (v0) at (-.7,-1);
  \coordinate (v1) at (-.3,-1);
  \coordinate (v2) at (.3,-1);
  \coordinate (v3) at (.7,-1);
  \draw[-triangle 60] (v) edge (v0) edge (v1) edge (v2) edge (v3);  
 \end{tikzpicture}
 }_{j\times}.
\end{equation}
\end{prop}

\begin{proof}
(Sketch) For a graph $\Gamma$ we define an \emph{antenna} as an external leg or univalent vertex, together with a maximal adjacent string of bivalent vertices, see Figure \ref{fig:antennas} for an illustration. We call the non-antenna part of the graph the \emph{core}. Note that the core can be empty, and that happens iff the graph is a string of $\leq 2$-valent vertices.
We define a descending complete filtration on $\hGCor_n$ by the number of non-antenna vertices in graphs. The differential on the associated graded, say $\delta'$, sees only the terms that leave the number of core vertices the same.
% in which case it also leaves the core the same. 
In \eqref{equ:hGCodifferential} above, the first term, the second term for $j=1,2$ and the third term for $j=1$ can contribute.
Let us compute the cohomology. The differential $\delta'$ leaves the core intact, so the complex splits as a product of subcomplexes, one for each core.
We treat two cases:

(A) The core is empty. In this case the graph is a string of $\leq 2$-valent vertices. It must have one of the following forms:
\begin{align*}
C_{i,j}&= 
\begin{tikzpicture}[baseline=-.65ex]
  \node[int] (v1) at (0,0) {};
  \node (v2) at (1,0) {$\cdots$};
\node[int] (v3) at (2,0) {};
\node (v4) at (3,0) {$\cdots$};
\node[int] (v5) at (4,0) {};
\coordinate (w1) at (-1,0);
\coordinate (w2) at (5,0);
\draw[-triangle 60] (v1) edge (w1) (v2) edge (v1) (v3) edge (v2) edge (v4) (v4) edge (v5) (v5) edge (w2);
 \end{tikzpicture}
=\pm C_{j,i}
\\
D_i &=
\begin{tikzpicture}[baseline=-.65ex]
  \node[int] (v1) at (0,0) {};
  \node (v2) at (1,0) {$\cdots$};
\node[int] (v3) at (2,0) {};
\coordinate (w2) at (3,0);
\draw[-triangle 60] (v1) edge (v2) (v2) edge (v3) (v3) edge (w2);
 \end{tikzpicture}
\end{align*}
Here the subscripts $i$ and $j$ refer to the lengths of the directed substrings, counted in edges.
One checks that the differential has the following form:
\begin{align*}
 \delta' C_{i,j} &=  \pm \epsilon_i C_{i+1,j} \pm C_{i,j+1} \\
 \delta' D_{i} &=  \pm \epsilon_{i+1} D_{i+1} \pm C_{1,i+1} 
\end{align*}
where $\epsilon_{i}$ is $0$ for $i$ odd and $1$ for $i$ even.
One easily checks that the cohomology of the resulting complex is one-dimensional, the cohomology class being represented by
\[
 C_{1,1} = 
\begin{tikzpicture}[baseline=-.65ex]
  \node[int] (v1) at (0,0) {};
\coordinate (w2) at (1,0);
\coordinate (w1) at (-1,0);
\draw[-triangle 60] (v1) edge (w2) edge (w1);
 \end{tikzpicture}
\]

This class is the image of the cocycle \eqref{equ:singleclass}. 
Note also that the cocycle \eqref{equ:singleclass} cannot be exact since it consists of graphs with one vertex only.

(B) The core is not empty. In this case the subcomplex corresponding to the core has the form $(\otimes_j C_j)^G$, where $j$ runs over the vertices in the core, $G$ is the symmetry group of the core and $C_j$ is a complex that models the antennas that can be attached at vertex $j$. In fact the complexes $C_j$ can be of one of five different forms, depending on whether $j$ is zero- one- or $\geq 2$-valent in the core and on whether or not it has an outgoing edge in the core.
One may compute the cohomology of each of these complexes.
Let us introduce an auxiliary complex $(A,d)$ which models a single antenna. It has elements $F_{i,j}$, $i\geq 1, j\geq 0$ and $G_i$, $i\geq 1$. They stand for the following antennas:
\begin{align*}
F_{i,j}&= 
\begin{tikzpicture}[baseline=-.65ex]
  \node[int] (v1) at (0,0) {};
  \node (v2) at (1,0) {$\cdots$};
\node[int] (v3) at (2,0) {};
\node (v4) at (3,0) {$\cdots$};
\node[int] (v5) at (4,0) {};
\coordinate (w1) at (-1,0);
\node[ext] (w2) at (5,0) {core};
\draw[-triangle 60] (v1) edge (w1) (v2) edge (v1) (v3) edge (v2) edge (v4) (v4) edge (v5) (v5) edge (w2);
 \end{tikzpicture}
\\
G_i &=
\begin{tikzpicture}[baseline=-.65ex]
  \node[int] (v1) at (0,0) {};
  \node (v2) at (1,0) {$\cdots$};
\node[int] (v3) at (2,0) {};
\node[ext] (w2) at (3,0) {core};
\draw[-triangle 60] (v1) edge (v2) (v2) edge (v3) (v3) edge (w2);
 \end{tikzpicture}
\end{align*}
where $i,j$ are the length of the directed strings, in edges.
The differential is 
\begin{align*}
 d F_{i,j} &=  \pm \epsilon_i F_{i+1,j} \pm  F_{i,j+1} \\
 d G_{i} &=  \pm \epsilon_{i+1} G_{i+1} \pm  F_{1,i+1} 
\end{align*}
One checks easily that the complex $A$ is acyclic, i.~e., $H(A,d)=0$. 
Now consider one of the five cases above, an at least bivalent core vertex with outgoing edge in the core. Then $C_j\cong (S(A), \delta'=d+d')$ where $S(\dots)$ denotes the completed symmetric product, $d$ is the differential induced by $d$ on $A$ and $d'$ is multiplication by $F_{1,1}$.
Using that $H(A,d)=0$ one can see that there is a single cohomology class represented by $1+\cdots \in S(A)$.
Here the ``1'' corresponds to the no-antenna configuration.

Next consider an at least bivalent core vertex $j$, with no outgoing edge in the core. Then the complex $C_j$ is a quotient of $(S^+(A), \delta'=d+d')$, obtained by modding out all monomials that do not contain at least one $F_{i,0}$, $i\geq 1$, where $S^+(A)$ denotes the completed symmetric product, without constant term.
Since $S^+(A)$ is acyclic (as one may see using acyclicity of $A$), we can alternatively compute the cohomology of the subspace $S^+(A')$, where $A'\subset A$ is obtained by removing the $F_{i,0}$. To compute $H(A')$ we might as well compute $H(A/A')$, which is easily checked to be one-dimensional, the cohomology class represented by $F_{1,0}$.
Hence one finds $H(A')$ is one-dimensional, a representative of the cohomology class is $F_{1,1}$, living in degree 1. Hence $H(S^+(A'), d)\cong S^+(F_{1,1} \K)\cong F_{1,1}\K$ is one dimensional.
By a standard spectral sequence argument, so is $H(C_j)$.

Next, consider the case of vertex $j$ one-valent in the core, with the incident edge outgoing.
The complex we need to consider is $(S^{\geq 2}(A), \delta'=d+d')$, where $S^{\geq 2}(A)$ is the completed symmetric product without constant and linear term.
Since $A$ is acyclic $H(C_j)=0$. 

The case of one-valent vertex $j$ with incoming edge in the core is handled similarly to the above discussion, and one finds that $H(C_j)=0$ as well, since $H(S^{\geq 2}(A'), d)\cong S^{\geq 2}(F_{1,1} \K)=0$.

Finally consider $j$ zero-valent in the core. Then $C_j$ is a quotient of $S^{\geq 3}(A)$, which can be seen to be acyclic since $H(S^{\geq 3}(A'), d)\cong S^{\geq 3}(F_{1,1} \K)=0$.

To summarize, only graphs contribute for which all core vertices are at least two-valent, and each such core contributes one or no class, depending on whether the core has odd symmetries. One checks that the contributing (isomorphism classes of) graphs are identified with the image of $\GCo_n$ under the map \eqref{equ:hairymap}, plus the extra class identified in step (A).
%For $j$ univalent there is no cohomology. Hence graphs with univalent vertices in the core do not contribute. 
%For $j$ a $\geq 2$-valent vertex there is a single cohomology class in $H(C_j)$. The representing cocycle is just the image of a graph in $\GCo_n$ under the map \eqref{equ:hairymap}. Furthermore for any graph in $\GCo_n$ there is a non-vanishing class in the associated graded we consider.
Hence the cohomology of the associated graded may be identified with the $\GCo_n$ under \eqref{equ:hairymap}, plus the one class.
It follows by a standard spectral sequence argument that the inclusion \eqref{equ:hairymap} is a quasi-isomorphism, up to this class.

%
%We define the core of a graph $\Gamma$ as the graph obtained by (i) removing external edges (ii) iteratively removing all one-valent vertices.
%One considers a spectral sequence such that the first differential leaves the core fixed.
%There are two cases to consider: (A) Suppose the core is non-empty. Then there is an explicit contracting homotopy. (Contract the outgoing edge at a core vertex if there is exactly one edge which is outgoing, and send the graph to zero otherwise.)
%
%(B) Suppose the core is empty. The graphs with empty core (i.e., trees) span a subcomplex inside $\hGCor_n$. 
% To compute its cohomology one takes a spectral sequence whose first differential does not create any new external edges.
%The cohomology is trivial, up to one class.% (TODO: check this).
\end{proof}

\begin{figure}
\[
\begin{tikzpicture}[
vert/.style={circle,draw,fill, minimum size=5pt, inner sep=0}, invvert/.style={inner sep=-1,minimum size=-1}, 
ext/.style={circle,draw, minimum size=5pt, inner sep=0},redvert/.style={circle,draw=gray,fill=gray, minimum size=5pt, inner sep=0},
invvert/.style={inner sep=-1,minimum size=-1}, every edge/.style={draw, -triangle 60},
scale=2 ]
\node (v0) at (0.9,1) [vert] {};
\node (v1) at (1.3,0.7) [vert] {};
\node (v2) at (1.9,0.9) [vert] {};
\node (v3) at (1.5,1.4) [vert] {};
\node (v4) at (2,1.6) [vert] {};
\node (v5) at (2.4,1.9)[ redvert] {};
\node (v6) at (2.6,1.4) [redvert] {};
\node (v7) at (1.2,1.7) [redvert] {};
\node (v8) at (0.8,1.8) [redvert] {};
\node (v9) at (2.7,0.9) [redvert] {};
\node (v10) at (1.3,0.4) [redvert] {};
\node (v11) at (1.1,0.2) [redvert] {};
\node (v12) at (0.4,1.6) [redvert] {};
\draw (v0)edge(v1) (v1)edge[gray](v10) (v1)edge(v2) (v0)edge(v3) (v3)edge(v1) (v3)edge(v2) (v2)edge(v4) (v4)edge[gray](v5) (v4)edge[gray](v6) (v7)edge[gray](v3);
\draw[gray] (v6)edge(v9) (v7)edge(v8) (v8)edge(v12) (v10)edge(v11);
\draw[gray] (v3) edge (1.5,2);
\draw[gray] (v4) edge (2,2);
\draw[gray] (v12) edge (0,1.5);
\draw[gray] (v11) edge (0.5,0);
\end{tikzpicture}
\]
\caption{\label{fig:antennas} The antennas of the graph are drawn in gray. The non-antenna-part is the core of the graph.}
\end{figure}
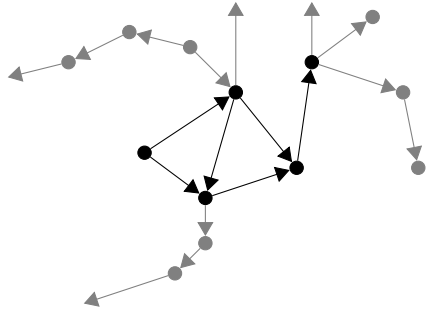

%\begin{rem}
% A similar statement probably holds true for the inclusion $\iGCor_3 \to \hGCor_3$, but it is not needed here.
%\end{rem}

\begin{rem}
 The map from $\GCo_3$ to the bi-Chevalley complex of any Lie bialgebra factors through $\hGCor_3$. In fact, both complexes $\GCo_3$ and $\hGCor_3$ can be thought of as universal versions of the bi-Chevalley complex.
\end{rem}

\subsection{Kontsevich's operads $\Graphs_n$}
M. Kontsevich \cite{K2} defined operads $\Graphs_n$ such that $\Graphs_n(N)$ is spanned by isomorphism classes of undirected graphs with vertices of two sorts: There are $N$ numbered \emph{external} vertices and and arbitrary (finite) number of unlabelled \emph{internal} vertices of valence at least 3. One furthermore requires that each connected component contains at least one external vertex. The operadic compositions are defined as insertions at external vertices, the differential is vertex splitting, pictorially
\begin{align*}
\delta
 \begin{tikzpicture}[baseline=-.65ex]
  \node[ext] (v) at (0,0) {j};
  \coordinate (w1) at (-.6,.5);
\coordinate (w2) at (-.2,.5);
\coordinate (w3) at (.2,.5);
\coordinate (w4) at (.6,.5);
\node at (0,.8){$\cdots$};
\draw (v) edge (w1) edge (w2) edge (w3) edge (w4);
 \end{tikzpicture}
&=
\sum
\begin{tikzpicture}[baseline=-.65ex]
  \node[ext] (v) at (0,0) {j};
\node[int] (w) at (0,.6) {};
  \coordinate (w1) at (-.6,.5);
\coordinate (w2) at (-.2,1);
\coordinate (w3) at (.2,1);
\coordinate (w4) at (.6,.5);
\node at (0,1.2){$\cdots$};
\draw (v) edge (w1) edge (w4) edge (w) (w) edge (w3) edge (w2);
 \end{tikzpicture}
&
\delta
 \begin{tikzpicture}[baseline=-.65ex]
  \node[int] (v) at (0,0) {};
  \coordinate (w1) at (-.6,.5);
\coordinate (w2) at (-.2,.5);
\coordinate (w3) at (.2,.5);
\coordinate (w4) at (.6,.5);
\node at (0,.8){$\cdots$};
\draw (v) edge (w1) edge (w2) edge (w3) edge (w4);
 \end{tikzpicture}
&=
\sum
\begin{tikzpicture}[baseline=-.65ex]
  \node[int] (v) at (0,0) {};
\node[int] (w) at (0,.6) {};
  \coordinate (w1) at (-.6,.5);
\coordinate (w2) at (-.2,1);
\coordinate (w3) at (.2,1);
\coordinate (w4) at (.6,.5);
\node at (0,1.2){$\cdots$};
\draw (v) edge (w1) edge (w4) edge (w) (w) edge (w3) edge (w2);
 \end{tikzpicture}.
\end{align*}
 
An example of a graph occurring in $\Graphs_n$ with some indication of the sign rules used can be found in Figure \ref{fig:graphsexample}. For a more detailed discussion we refer the reader to \cite{K2} and \cite{grt}.

\begin{figure}
\centering
\begin{tikzpicture}
\node [ext] (v6) at (-4.5,-2) {4};
\node [ext] (v5) at (-6.5,-2) {3};
\node [ext] (v7) at (-5.5,-2) {2};
\node [ext] (v4) at (-7.5,-2) {1};
\node [int] (v2) at (-7,-1) {};
\node [int] (v1) at (-6.5,0) {};
\node [int] (v3) at (-6,-1) {};
\draw  (v1) edge (v2);
\draw  (v2) edge (v3);
\draw  (v1) edge (v3);
\draw  (v2) edge (v4);
\draw  (v2) edge (v5);
\draw  (v3) edge (v5);
\draw  (v1) edge (v6);
\draw  (v7) edge (v6);
\draw  (v3) edge (v7);
\node [ext] (v9) at (-3,-2) {1};
\node [ext] (v11) at (-2,-2) {2};
\node [ext] (v14) at (-1,-2) {3};
\node [int] (v12) at (-2.5,0) {};
\node [int] (v10) at (-2,-1) {};
\node [int] (v8) at (-2.5,-0.5) {};
\node [int] (v13) at (-1.5,-0.5) {};

\draw  (v10) edge (v11);
\draw  (v8) edge (v10);
\draw  (v12) edge (v8);
\draw  (v8) edge (v13);
\draw  (v12) edge (v13);
\draw  (v13) edge (v10);
\draw  (v10) edge (v9);
\draw  (v10) edge (v14);
\node [int] (v15) at (-1.5,0) {};
\draw  (v12) edge (v15);
\draw  (v15) edge (v13);
\draw  (v8) edge (v15);
\end{tikzpicture}
\caption{\label{fig:graphsexample} Two examples of graphs in $\Graphs_n$. The right hand graph is zero for $n$ even because it has an odd symmetry, i.~e., a symmetry that acts by an odd permutation on the set of edges. }
\end{figure}
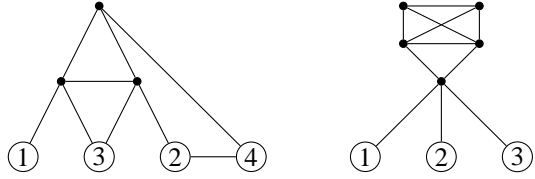

There are two important facts we will need about the operads $\Graphs_n$.
\begin{itemize}
 \item There is an inclusion $\e_n\to \Graphs_n$ which is a quasi-isomorphism, see \cite{LV, K2}.
 \item There is an action of the dg Lie algebra $\GC_n$ on $\Graphs_n$ by operadic derivations, see \cite{grt}.
\end{itemize}

In fact, for our application is is more convenient to enlarge the operad by merely dropping the condition that internal vertices have to be at least trivalent.
We call the resulting operad $\fcGraphs_n$ and cite without proof the following facts \cite{grt}:
\begin{itemize}
 \item There natural inclusion $\Graphs_n\to \fcGraphs_n$ is a quasi-isomorphism.
 \item The action of $\GC_n$ on $\Graphs_n$ extends to an action of $\fGCc_n$ on $\fcGraphs_n$ by operadic derivations.
\end{itemize}

\subsection{The oriented version $\fcGraphso_n$}
Again we may modify M. Kontsevich's construction slightly by allowing directed edges and internal vertices of all valences to obtain an operad $\dfcGraphs_n$.
One has a quasi-isomorphism $\Graphs_n\to\dfcGraphs_n$, so not much has changed, in particular $H(\dfcGraphs_n)=\e_n$.
However, we may pass to a sub-operad
\[
 \op P %\Graphso_n
 \subset \dfcGraphs_n
\]
whose elements are series of graphs such that 
\begin{itemize}
 \item There are no directed cycles.
 \item There are no edges starting at the external vertices.
\end{itemize}

We may then pass to a quotient $\fcGraphso_n$ of $\op P$, defined by setting all graphs to zero that contain internal vertices without outgoing edges, or external vertices with outgoing edges. Finally we may define the suboperad $\Graphso_n\subset \fcGraphso_n$ by requiring all internal vertices to be of valence $\geq 2$.

Examples and non-examples are shown in Figure \ref{fig:graphsex}.
%The operad structure on $\Graphso$ is by insertion at external vertices and reconnecting the dangling edges in all possible ways, cf. \cite{}.
%The differential is by vertex splitting.

\begin{figure}
\centering
\begin{tikzpicture}[every edge/.style={draw, -triangle 60},scale=1]
\node [ext] (v5) at (-2,0) {1};
\node [ext] (v4) at (-1,0) {2};
\node [ext] (v6) at (0,0) {3};
\node [int] (v2) at (-1.5,1) {};
\node [int] (v1) at (-1,2) {};
\node [int] (v3) at (-0.5,1) {};
\draw  (v1) edge (v2);
\draw  (v1) edge (v3);
\draw  (v3) edge (v4);
\draw  (v2) edge (v4);
\draw  (v2) edge (v5);
\draw  (v3) edge (v6);
\node [ext] (v11) at (1,0) {1};
\node [ext] (v12) at (2,0) {2};
\node [ext] (v16) at (3.5,0) {1};
\node [int] (v7) at (-1,3) {};
\node [int] (v9) at (1.5,1) {};
\node [int] (v8) at (1,2) {};
\node [int] (v10) at (2,2) {};
\node [int] (v13) at (3.5,1) {};
\node [int] (v14) at (4.5,0.5) {};
\node [int] (v15) at (4.5,1.5) {};
\draw  (v7) edge[bend left] (v1);
\draw  (v7) edge[bend right] (v1);
\draw  (v9) edge (v8);
\draw  (v10) edge (v9);
\draw  (v9) edge (v11);
\draw  (v8) edge (v10);
\draw  (v10) edge (v12);
\draw  (v13) edge (v14);
\draw  (v15) edge (v14);
\draw  (v15) edge (v13);
\draw  (v13) edge (v16);
\end{tikzpicture}
\caption{\label{fig:graphsex} The left graph is an example of a graph in $\Graphso$. The middle graph is not admissible because it contains an oriented cycle. The right hand graph is not admissible because it contains an internal vertex without outgoing edges. }
\end{figure}
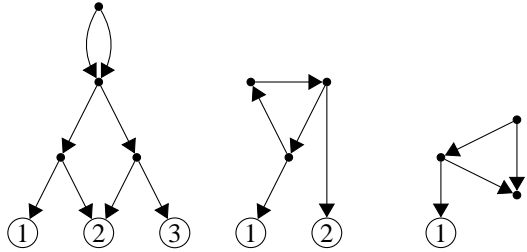

There is an injective operad map 
\[
 \te_{n-1} \to \Graphso_n
\]
sending the commutative product operation $m_2\in \te_{n-1}(2)$ to the graph 
\begin{equation}\label{equ:graphproduct}
 \begin{tikzpicture}[scale=.5]
  \node[ext] at (0,0) {};
  \node[ext] at (1,0) {};
 \end{tikzpicture}
\end{equation}
and the $\hoLie_n$ operations $\mu_N\in\te_{n-1}(N)$ ($N\geq 2$) to the graphs
\begin{equation}\label{equ:graphfork}
%\frac 1 {N!}
\underbrace{
  \begin{tikzpicture}[scale=1]
  \node[int] (v) at (0,1) {};
  \node[ext] (v0) at (-1.3,0) {1};
  \node[ext] (v4) at (1.3,0) {N};
  \node[ext] (v1) at (-.7,0) {2};
  \node (v2) at (0,0) {$\cdots$};
  \node[ext] (v3) at (.7,0) {\tiny{N-1}};
  \draw[-triangle 60] (v) edge (v1)  edge (v3) edge (v0) edge (v4);
 \end{tikzpicture}}_{N \times}
\end{equation}

\begin{thm}\label{thm:enandGraphsor}
 The map $\te_{n-1} \to \Graphso_n$ above is a quasi-isomorphism.
\end{thm}
In other words $\Graphso_n$ is quasi-isomorphic to $\hoe_{n-1}$.
\begin{proof}
% Note that there is a natural surjective map of operads
% \[
%  \tGer \to \Ger
% \]
%owed to the fact that any Gerstenhaber algebra is a $\tGer$ algebra.
Note that there is a natural surjective map of $\mathbb{S}$-modules of complexes
\[
 \Graphso_n\to \te_{n-1}
\]
sending all graphs that contain a vertex with more then one incoming edge to 0. Note that this is not a map of operads. Nevertheless it is a one-sided inverse to the natural map $\te_{n-1}\to \Graphso_n$. It follows that the identity map $\e_{n-1} \to \e_{n-1}$ may be written as
\[
 \e_{n-1} = H(\te_{n-1}) \to H(\Graphso_n) \to H(\te_{n-1}) = \e_{n-1}
\]
In particular, the map $H(\te_{n-1}) \to H(\Graphso_n)$ is injective. 

Moreover $\Graphso_n$ splits as complexes
\[
 \Graphso_n \cong \te_{n-1} \oplus {\Graphso_n}'
\]
where ${\Graphso_n}'$ is spanned by graphs with at least one vertex with more than one incoming edge. Let us call such vertices \emph{bad vertices}, and a graph containing a bad vertex a \emph{bad graph}. 
To prove the Theorem, we need to show that ${\Graphso_n}'$ is acyclic.

Note that for each bad graph there is an external vertex with the lowest number (say $j$), which has a bad vertex as ancestor. Let us call $j$ the bad index. There is also a unique string of edges connecting the vertex $j$ to the bad vertex. Let us call this string the bad string. See Figure \ref{fig:badstring} for an illustration of these concepts.
Note that at this point we use that each internal vertex has at least one outgoing edge.

The differential may (i) increase the bad index or (ii) increase the length of the bad string.
One can set up a spectral sequence on the length of the bad string, such that the first page differential increases this length by one (and leaves the bad index the same).
The resulting complex is acyclic. (It looks like the bar resolution of a free cocommutative coalgebra, see Figure \ref{fig:badstring}.) It follows that ${\Graphso_n}'$ is acyclic as claimed.
%\hfill \qed
 
\end{proof}

\begin{figure}
\centering
\begin{tikzpicture} 
\node [int] (v1) at (-0.5,3.5) {};
\node [int] (v2) at (0,4.5) {};
\node [int] (v4) at (0.5,3) {};
\node [int,red] (v7) at (0.5,2) {};
\node [int] (v3) at (1.5,3.5) {};
\node [ext] (v6) at (-1,4.5) {1};
\node [ext] (v5) at (-1,2.5) {2};
\node [ext] (v9) at (0,1.5) {3};
\node [ext] (v8) at (1,1.5) {4};
\draw [-triangle 60] (v2) edge (v1);
\draw [-triangle 60] (v2) edge (v3);
\draw [-triangle 60] (v1) edge (v4);
\draw [-triangle 60] (v4) edge (v3);
\draw [-triangle 60] (v1) edge (v5);
\draw [-triangle 60] (v2) edge (v6);
\draw [-triangle 60,red] (v4) edge (v7);
\draw [-triangle 60] (v7) edge (v8);
\draw [-triangle 60,red] (v7) edge (v9);
\draw [-triangle 60] (v3) edge (v8);
\draw [-triangle 60] (v2) edge (v4);
\end{tikzpicture}
\caption{\label{fig:badstring} An admissible graph, with bad index 3. The bad string is drawn in red. }
\end{figure}
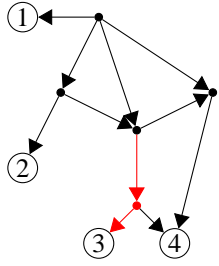

\begin{rem}
 The operad $\fcGraphso_3$ is defined so that it acts naturally on $S\alg g[-1]$, for $\alg g$ an $\infty$-Lie bialgebra. In fact, it can be considered as one version of the operad of natural operations on $S\alg g[-1]$.\footnote{There is an even more natural version which has incoming legs in addition. However, we stick to the simpler one above.} The action of $\te_2$ on $S\alg g[-1]$ naturally factorizes through $\fcGraphso_3$.
\end{rem}

Finally an important fact is that there is an action of $\fGCco_n$ on $\fcGraphso_n$ by operadic derivations. It is constructed as follows. First note that by generalities on operadic twisting (cf. \cite{grt} and \cite{vastwisting}) one has an action of $\dfGCc_n$ on $\dfcGraphs_n$. This action cannot create directed cycles in graphs if there were none before in the acting graph or the acted upon graph. Hence the action restricts to an action of $\fGCco_n$ on $\op P$.
Next, if a graph occurring in an element of $\op P$ has an internal vertex without outgoing edges, then so do the graphs produced after acting with an element of $\fGCco_n$. The same holds for external vertices with outgoing edges. Hence the action descends to the quotient so that we obtain the desired action of $\fGCco_n$ on $\fcGraphso_n$ by operadic derivations. The action cannot create uni-valent vertices in graphs if there were none before, so it restricts naturally to an action of $\GCo_n$ on $\Graphso_n$.

\begin{rem}\label{rem:explaction}
Explicitly, the formula for the action has the following form. Suppose we have a graph $\gamma\in \fGCco_{n}$, and a graph $\Gamma\in \fcGraphso_n(N)$. 
Let us form the graph $\gamma_1\in \fcGraphso_n(1)$ by declaring  one vertex external. Note that if $\gamma$ has more then one vertex without outgoing edges then $\gamma_1=0$, and if it has exactly one such vertex, then $\gamma_1$ is obtained by declaring that vertex external.
Then, disregarding the precise signs and prefactors the action is 
\[
\gamma \cdot \Gamma = \sum_v \pm \Gamma \bullet_v \gamma
+
\gamma_1\circ_1 \Gamma + \sum_j \pm \Gamma \circ_j \gamma_1.
\]
Here $\circ_j$ denotes the operadic composition as usual and the first sum is over internal vertices of $\Gamma$, with $\bullet_v\gamma$ denoting the operation of inserting $\gamma$ in place of $v$.
\end{rem}

\section{The cohomology of the loopless graph complex, and the proof of Theorem \ref{thm:main}}\label{sec:theproof}
\subsection{Recollections from \cite{grt}}
The proof of Theorem \ref{thm:main} will make essential use of the results and techniques of \cite{grt}.
Recall from section \ref{sec:graphcomplexes} above that there is an action of the graph complex $\fGCc_n$ by operadic derivations on the operad $\fcGraphs_n$, which is in turn quasi-isomorphic to $\e_n$.
From this it follows that there is a map of Lie algebras
\begin{equation}
H(\fGCc_n)\to H(\Der(\hoe_n))
\end{equation}
where $\Der(\hoe_n)$ is the complex of derivations of the operad $\hoe_n$, see \cite[section 2]{grt} for its definition.
Concretely, the map can be realized as follows. There are maps of complexes
\[
\Der(\hoe_n) \to \Def(\hoe_n\to \fcGraphs_n)[1] \leftarrow \fGCc_n
\]
where $\Def(\hoe_n\to \fcGraphs_n)$ is the operadic deformation complex of the composition $\hoe_n\to \e_n\to \fcGraphs_n$, see again \cite[section 2]{grt} for the definition.
In fact, it turns out that one may identify a subcomplex of ``connected'' elements (see \cite[section 4]{grt})
\[
\Def(\hoe_n\to \fcGraphs_n)_\conn\subset \Def(\hoe_n\to \fcGraphs_n)
\]
into which the image of $\fGCc_n$ falls, and the cohomology of the full complex is a symmetric product of that of the connected part.
One main result of \cite{grt} is that the induced map from the graph cohomology to the cohomology of the connected part is an isomorphism, up to one class, i.~e.,
\[
\K \oplus H(\fGCc_n) \cong H(\Def(\hoe_n\to \fcGraphs_n)_\conn[1]).
\]
The extra class has a nice interpretation: One may enlarge the Lie algebra structure on the graph complex to a semidirect product $\K \ltimes \fGCc_n$, where the extra generator acts on a graph by multiplication with the first Betti number.

\subsection{The proof of Theorem \ref{thm:main}}
One can mimick the construction outlined in the last section, replacing the graph complex $\fGCc_n$ by the oriented graph complex $\GC_{n+1}^{or}$, and replacing the operad $\fcGraphs_n$ by the operad $\Graphso_{n+1}$, which is also a model for the $\e_n$ operad by Theorem \ref{thm:enandGraphsor}. 
Since there is an action of $\GC_{n+1}^{or}$ on $\Graphso_{n+1}$ by operadic derivations we obtain a map of complexes 
\[
\GCo_{n+1} \to \Def(\hoe_n\to \Graphso_{n+1})[1]. 
\]
As in \cite[section 4]{grt}, we may identify a ``connected'' subcomplex of the complex on the right hand side, into which the image of the graph complex falls. 
\[
\GCo_{n+1} \to \Def(\hoe_n\to \Graphso_{n+1})_\conn[1]. 
\]
Furthermore, there is one special class in the complex on the right hand side. The underlying derivation of $\e_n$ rescales the Lie bracket.
Concretely, the map sends the $\hoLie_n$ operations $\mu_N\in \hoe_n(N)$ ($N\geq 2$) to multiples of the graphs \eqref{equ:graphfork} above, and all other generators to zero.
We denote this class $T$ for concreteness.
The main step in the proof of Theorem \ref{thm:main} is to show the following Proposition.
\begin{prop}\label{prop:premain}
The map 
\[
\K T \oplus \GCo_{n+1} \to \Def(\hoe_n\to \Graphso_{n+1})_\conn[1]. 
\]
is a quasi-isomorphism of complexes.
% Furthermore the induced map 
% \[
% \K \ltimes H(\GCo_{n+1}) \to H(\Der(\hoe_n)) 
% \]
% is an injective map of Lie algebras.
\end{prop}

Let us use this result to prove Theorem \ref{thm:main}.
\begin{proof}[Proof of Theorem \ref{thm:main}]
We have the following zig-zag of quasi-isomorphisms of complexes
\begin{multline*}
\K \oplus \GCo_{n+1}\to  
\Def(\hoe_{n} \to \Graphso_{n+1})_\conn[1]
\leftarrow 
\Def(\hoe_{n} \to \te_{n})_\conn[1]
\\
\to \Def(\hoe_{n} \to \e_{n})_\conn[1]
 \to
 \Def(\hoe_{n} \to \fcGraphs_{n})_\conn[1]
\leftarrow 
\K \oplus \fGCc_{n}.
\end{multline*}
Here the leftmost arrow is a quasi-isomorphism because of Proposition \ref{prop:premain}, the next two arrows because of Theorem \ref{thm:enandGraphsor} and Proposition \ref{prop:tenenqiso}, combined with Lemma \ref{lem:defpresqiso}. The fact that the last two arrows are quasi-isomorphisms is shown in \cite{grt}.
It follows in particular that 
\[
\K \oplus H(\GCo_{n+1}) \cong \K \oplus H(\fGCc_{n})
\]
as graded vector spaces. 
All complexes in the above zig-zag can be endowed with an extra grading by the first Betti number of graphs, preserved by the maps and differentials. Hence the identification above preserves that grading. The two copies of $\K$ are the only pieces with Betti number 0, while the remainder has Betti numbers $\geq 1$. Hence the two ``extra classes'' are mapped onto each other, and so are $H(\GCo_{n+1})$ and $H(\fGCc_{n})$
so that one has an isomorphism of graded vector spaces
\[
H(\GCo_{n+1}) \cong H(\fGCc_{n}).
\]
We still have to check that this map preserves the Lie brackets.
First, since $\GCo_{n+1}$ acts on a model for $e_n$ by operadic derivations, the inclusion 
\[
 \K \ltimes H(\GCo_{n+1}) \to H(\Der(\hoe_n)) 
\]
preserves the Lie bracket. Alternatively, one can give a slightly more detailed proof along the lines of the proof of Proposition 5.4 of \cite{grt}.
Hence we have maps of Lie algebras 
\[
\K \ltimes H(\GCo_{n+1}) \to H(\Der(\hoe_n)) 
\leftarrow
\K \ltimes H(\fGCc_{n}).
\]
Both maps have the same image and are isomorphisms of Lie algebras onto their images, so both Lie algebras must be isomorphic.
\end{proof}

It remains to show the Proposition.
\begin{proof}[Proof of Proposition \ref{prop:premain}]
The proof follows closely that of Theorem 1.3 in \cite{grt}.
Recall that the map 
\[
\hoe_n \to \Graphso_{n+1}
\] 
sends the product generator to the graph with two vertices and no edges \eqref{equ:graphproduct}, and the $\hoLie_n$ generator $\mu_N$ to the graph \eqref{equ:graphfork}.
Accordingly, the differential on  $\Def(\hoe_n \to \Graphso_{n+1})$ may be split into three parts
\[
d = \delta + d_\wedge + d_C
\]
where $\delta$ comes from the differential on $\Graphso_{n+1}$, $d_\wedge$ comes from the piece of the map involving the product generator and $d_C$ from the piece involving the $\hoLie_n$ generators.
We will use the same symbols to denote the respective pieces of the differential on the subcomplex  $\Def(\hoe_n \to \Graphso_{n+1})_\conn$.
There are several filtrations on that complex, coming from the number of internal or external vertices and the number of edges of graphs in $\Graphso_{n+1}$. Consider the decreasing complete filtration by the number of edges, then the differential on the associated graded is given by $d_\wedge$, which does not create edges, while $\delta$ and $d_C$ do not contribute since they create $\geq 1$ edges in graphs. As in loc. cit. the cohomology wrt. $d_\wedge$ may be computed and can be identified with the subcomplex of  $\Def(\hoe_n \to \Graphso_{n+1})_\conn$ consisting of maps that send all generators to zero except the $\hoLie_n$ generators $\mu_N$ ($N\geq 1$), and for which the image of each $\mu_N$ is a graph whose external vertices have valence one.
As those elements form a subcomplex, say 
\[
C\subset \Def(\hoe_n \to \Graphso_{n+1})_\conn
\]
we know that the inclusion is a quasi-isomorphism by a standard spectral sequence argument.
That subcomplex furthermore is isomorphic to the complex $\hGCor_{n+1}$ from section \ref{sec:beardGC}. Proposition \ref{prop:GChGC} then shows that its cohomology is that of the graph complex $\GCo_{n+1}$, up to the one class $T$ described above. 
To show the first statement of the Proposition we yet have to verify that the inclusion $H(\fGCco_{n+1})\to H(\Def(\hoe_n \to \fcGraphso_{n+1})_\conn)$ just discussed agrees with the map $H(\GCo_{n+1})\to H(\Def(\hoe_n \to \Graphso_{n+1})_\conn)$ defined through the action of $\GCo_{n+1}$ on $\Graphso_{n+1}$. This is not totally obvious since the maps do not agree on chains. However, a small graphical calculation shows that for a graph cocycle $\Gamma$ the difference between the two maps on chains is the coboundary of an element $X_\Gamma\in \Def(\hoe_n \to \Graphso_{n+1})_\conn$. Concretely, $X_\Gamma$ sends the counit of $e_n^\vee$ to the graph $\Gamma_1\in \fcGraphso_{n+1}$ (see Remark \ref{rem:explaction} for this notation). %This finishes the first part of the Proposition. The second part can be obtained by copy-pasting the proof of Proposition XX of \cite{grt}.
\end{proof}

%
%We have the following zig-zag of quasi-isomorphisms
%\[
% \Def(\hoe_{n-1} \to \Graphso_n) \leftarrow \Def(\hoe_{n-1} \to \te_{n-1}) \to \Def(\hoe_{n-1} \to e_{n-1}) \to \Def(\hoe_{n-1} \to \Graphs_{n-1}).
%\]
%Each of these complexes may be written as symmetric power of its (suitable defined) connected part and we have a zig-zag of quasi-isomorphisms
%\[
% \Def(\Ger_\infty \to \Graphso)_\conn \leftarrow \Def(\Ger_\infty \to \tGer)_\conn \to \Def(\Ger_\infty \to \Graphs_2)_\conn.
%\]
%There is a ntural inclusion $\GCor_3\to \iGCor_3 \to \Def(\Ger_\infty \to \Graphso)_\conn$. So we have a zig-zag
%\[
% \GCor_3\to \Def(\Ger_\infty \to \Graphso)_\conn \leftarrow \Def(\Ger_\infty \to \tGer)_\conn \to \Def(\Ger_\infty \to \Graphs_2)_\conn \leftarrow \fGC_{2,\conn}[-1].
%\]
%
%The right-most map is a essentially a quasi-isomorphism, up to one class. The middle maps are quasi-isomorphisms. 
%\begin{thm}
%The left-most map is a quasi-isomorphism, up to one class (namely the one from eqn. \eqref{equ:singleclass}).
%\end{thm}
%\begin{proof}
%The argument is a copy of the one proving Proposition XX in \cite{}.
%\end{proof}
%\begin{cor}
% $H(\GCor_3)[1]\cong H(\fGC_{2,\conn})\cong \prod_{j=1,5,9,\dots}\R[2-j]\oplus H(\GC_2)$.
%\end{cor}

\begin{rem}\label{rem:wheelclasses}
The simplest classes in the graph complex $\fGCc_n$ are the wheel graphs, composed of $j$ two-valent vertices, where $j=1,5,9,\dots$ if $n$ is even and $j=3,7,11,\dots$ when $n$ is odd. For example the 3-wheel occurring in the odd case is the following graph.

\[
 \begin{tikzpicture}
  \node[int] (v0) at (0:1) {};
%\node[int] (v1) at (60:1) {};
\node[int] (v2) at (120:1) {};
%\node[int] (v3) at (180:1) {};
\node[int] (v4) at (240:1) {};
%\node[int] (v5) at (300:1) {};
\draw (v0) edge (v2) edge (v4) (v2) edge (v4);
 \end{tikzpicture}
\]
In $\GCo_{n+1}$ one has even length wheel graphs, with alternating edge directions. Length $2,6,10,\dots$ occur for $n$ even, while lengths $4,8,12,\dots$ occur for $n$ odd by symmetry reasons.  % replacing all edges by two edges and a vertex from which they emanate. % $\leftarrow \bullet \rightarrow$. 
For example, the following graph appears in $\GCo_{3}$
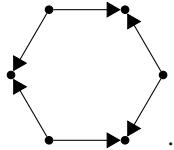
\[
 \begin{tikzpicture}
  \node[int] (v0) at (0:1) {};
\node[int] (v1) at (60:1) {};
\node[int] (v2) at (120:1) {};
\node[int] (v3) at (180:1) {};
\node[int] (v4) at (240:1) {};
\node[int] (v5) at (300:1) {};
\draw[-triangle 60] (v0) edge (v1) edge (v5) (v2) edge (v1) edge (v3) (v4) edge (v3) edge (v5);
 \end{tikzpicture}.
\]
In general the wheel of length $k$ in $\fGCc_n$ corresponds to the wheel of length $k+1$ in $\GCo_{n+1}$.
To see this, first note that the identification of $H(\fGCc_n)$ with $H(\GCo_{n+1})$ preserves the degree and genus. Hence the only candidate graphs to correspond to the $k$-wheel are $k+1$-wheels, with some orientation on the edges.
One may check that any class may be represented by graphs with no vertices with one incoming and one outgoing vertex. This leaves only one possible graph, the $k+1$-wheel with alternatingly oriented edges.
%In general, there are $n=2,4,6,\dots$ vertices and $2n$ edges. 
\end{rem}

\section{Application: Cycle-free formality morphisms and star products}
\label{sec:looplessdefq}
One application of the oriented graph complex $\GCo_2$ is in understanding the obstruction theory underlying infinite dimensional deformation quantization, as discussed in the introduction.
In this section, we will clarify a few points not treated there and finally show Corollary \ref{cor:looplessformality}.

First, one notes that one can show along the lines of \cite{vasilystable} that the obstruction theory for cycle-free universal formality morphisms is governed by $\fGCo_2$. We take this fact for granted. Using Theorem \ref{thm:main} one can see that in the relevant degrees 0 and 1 the cohomology agrees with that of the subcomplex $\GCo_2$, so we will take this complex as governing our deformation theory in the following.

As mentioned in the introduction, $H^1(\GCo_2)=H^1(\fGCo_2)$ is one-dimensional by Theorem \ref{thm:main}. Hence there is only one obstruction to the existence of cycle-free formality morphism, which however gets hit unfortunately. 

Another fact one learns from $H^1(\GCo_2)\cong \K$ is that the tangent space to the space of Maurer-Cartan elements modulo gauge at 0 is one-dimensional. This can be used to understand the space of Maurer-Cartan elements as a whole.

Note that $\GCo_2$ is graded by the first Betti numbers of graphs (-as are all relevant graph complexes the author is aware of-).
There are hence endomorphisms 
\[
\Phi_\lambda \colon \GCo_2 \to \GCo_2
\]  
rescaling a graph $\Gamma$ of genus $b_1$ by $\lambda^{b_1}$, for $\lambda\in \K$. In particular, for any (gauge non-trivial) Maurer-Cartan element $m$ we can define a family of Maurer-Cartan elements $m_\lambda=\Phi_\lambda(m)$, in particular $m_0=0$. It follows in particular that the space of Maurer-Cartan elements is connected.
Since we know that the tangent space at the identity is one-dimensional one may deduce that the space of Maurer-Cartan elements modulo gauge is a line.

Let us check that a non-trivial Maurer-Cartan element (say $m$) may be constructed inductively as
\[
m=m_4+m_5+\cdots
\]
where $m_k$ is a linear combination of graphs with $k$ vertices. 
Here $m_4$ is understood to be a multiple the Shoikhet cocycle depicted in Example \ref{ex:shoiobstruction}.
Possible obstructions to determining $m_k$ knowing $m_4,\dots, m_{k-1}$
fall into $H^2(\GCo_2)$, which may be identified with $H^2(\fGCc_1)\cong \K$ by Theorem \ref{thm:main}. The one possible obstruction class in $H^2(\fGCc_1)$ is the 3-wheel depicted in Remark \ref{rem:wheelclasses}. The corresponding class in $H^2(\GCo_2)$ has 4 vertices.
Note that the first (possible) obstruction that actually occurs is given by the bracket $[m_4,m_4]$ and is a linear combination of graphs with 7 vertices. %But the class of Remark \ref{rem:wheelclasses} has only 6 vertices. 
Hence no obstruction is hit, and a Maurer-Cartan element may be constructed step-by-step and rationally, thus showing one claim in Corollary \ref{cor:looplessformality}.

Next, let us fix one such Maurer-Cartan element $m$. It determines a non-standard $L_\infty$ structure on $\Tpoly$.  
Let us construct an $L_\infty$ quasi-isomorphisms
\[
 \Tpoly'\to \Dpoly
\]
given by universal cycle-free formulas using obstruction theory, where $\Tpoly'$ is $\Tpoly$ with the non-standard $L_\infty$ structure.
The obstructions again land in $H^1(\fGCo_2)$, which is one-dimensional. However, we may choose the Maurer-Cartan element (the prefactor in $m_4$ to be precise) so that the obstruction is not hit. Since there are no further cohomology classes in $H^1(\fGCo_2)$, the existence of the formality morphism is unobstructed after this.

Possible choices to be made during the construction are parameterized by $H^0(\fGCco_2)=H^0(\GC_1)=0$. Hence the formality morphism is unique up to gauge, and can be constructed by obstruction theory, hence rationally. This shows Corollary \ref{cor:looplessformality}.

\nocite{MerkulovPropProfile, MarklVoronov,MarklMerkulovShadrin}
\bibliographystyle{plain}
\bibliography{../biblio}

\begin{thebibliography}{10}

\bibitem{vogtmannhairy}
James Conant, Martin Kassabov, and Karen Vogtmann.
\newblock Hairy graphs and the unstable homology of {${\rm Mod}(g,s)$}, {${\rm
  Out}(F_n)$} and {${\rm Aut}(F_n)$}.
\newblock {\em J. Topol.}, 6(1):119--153, 2013.

\bibitem{cullervogtmann}
Marc Culler and Karen Vogtmann.
\newblock Moduli of graphs and automorphisms of free groups.
\newblock {\em Invent. Math.}, 84(1):91--119, 1986.

\bibitem{vasilystable}
Vasily Dolgushev.
\newblock {Stable Formality Quasi-isomorphisms for Hochschild Cochains I},
  2011.
\newblock arXiv:1109.6031.

\bibitem{vastwisting}
Vasily Dolgushev and Thomas Willwacher.
\newblock Operadic twisting -- with an application to deligne's conjecture.
\newblock 2012.

\bibitem{Kgrcomplex}
Maxim Kontsevich.
\newblock Formality {C}onjecture.
\newblock {\em Deformation Theory and Symplectic Geometry}, pages 139--156,
  1997.
\newblock D. Sternheimer et al. (eds.).

\bibitem{K2}
Maxim Kontsevich.
\newblock Operads and {M}otives in {D}eformation {Q}uantization.
\newblock {\em Lett. Math. Phys.}, 48:35--72, 1999.

\bibitem{K1}
Maxim Kontsevich.
\newblock Deformation quantization of {P}oisson manifolds.
\newblock {\em Lett. Math. Phys.}, 66(3):157--216, 2003.

\bibitem{LambrechtsTurchin}
Pascal Lambrechts and Victor Turchin.
\newblock Homotopy graph-complex for configuration and knot spaces.
\newblock {\em Trans. Amer. Math. Soc.}, 361(1):207--222, 2009.

\bibitem{LV}
Pascal Lambrechts and Ismar Volic.
\newblock Formality of the little {N}-disks operad, 2008.
\newblock arXiv:0808.0457.

\bibitem{lodayval}
J.-L. Loday and B.~Vallette.
\newblock {\em {Algebraic Operads}}.
\newblock Number 346 in {Grundlehren der mathematischen Wissenschaften}.
  {Springer, Berlin}, {2012}.
\newblock {to appear, online version at math.unice.fr/\textasciitilde
  brunov/Operads.pdf}.

\bibitem{MarklMerkulovShadrin}
M.~Markl, S.~Merkulov, and S.~Shadrin.
\newblock Wheeled {PROP}s, graph complexes and the master equation.
\newblock {\em J. Pure Appl. Algebra}, 213(4):496--535, 2009.

\bibitem{MarklVoronov}
M.~Markl and A.~A. Voronov.
\newblock P{ROP}ped-up graph cohomology.
\newblock In {\em Algebra, arithmetic, and geometry: in honor of {Y}u. {I}.
  {M}anin. {V}ol. {II}}, volume 270 of {\em Progr. Math.}, pages 249--281.
  Birkh\"auser Boston Inc., Boston, MA, 2009.

\bibitem{MerkulovPropProfile}
S.~A. {Merkulov}.
\newblock {PROP profile of deformation quantization and graph complexes with
  loops and wheels}.
\newblock {\em ArXiv Mathematics e-prints}, December 2004.

\bibitem{penkavavanhaecke}
Michael Penkava and Pol Vanhaecke.
\newblock Deformation quantization of polynomial {P}oisson algebras.
\newblock {\em J. Algebra}, 227(1):365--393, 2000.

\bibitem{carlothomas}
Carlo Rossi and Thomas Willwacher.
\newblock P.~{E}tingof's conjecture about {D}rinfeld associators.

\bibitem{shoikhetlinfty}
B.~{Shoikhet}.
\newblock An ${L}_\infty$ algebra structure on polyvector fields.
\newblock {\em ArXiv e-prints}, May 2008.

\bibitem{grt}
Thomas Willwacher.
\newblock {M. Kontsevich's graph complex and the Grothendieck-Teichm\"uller Lie
  algebra}, 2010.
\newblock arxiv:1009.1654.

\end{thebibliography}

\end{document}